\documentclass[11pt]{amsart}
\usepackage{url}
\usepackage{graphicx}
\usepackage{bbold}
\usepackage{latexsym}
\usepackage{amsfonts,amsmath,amssymb}
\newtheorem{theorem}{Theorem}[section]																
\newtheorem{lemma}[theorem]{Lemma}

\newcommand{\E}{\operatorname{\mathbb{E}}}
\newcommand{\Prob}{\operatorname{\Bbb{P}}}

\usepackage{color}

\def\eps{\varepsilon}
\def\la{\lambda}

\def\part{\partial}

\def\b1{\bold 1}

\newcommand{\beq}{\begin{equation}}
\newcommand{\eeq}{\end{equation}}

\theoremstyle{remark}

\numberwithin{equation}{section}
\date{\today}

\begin{document}

\title[Stable Matchings]{Stable matchings with switching costs}

\author {Boris Pittel and Kirill Rudov}
\address{Department of Mathematics, Ohio State University}
\email{pittel.1@osu.edu}
\address{Department of Economics, University of California, Berkeley}
\email{krudov@berkeley.edu}
\maketitle

\begin{abstract} In a stable matching problem there are two groups of agents, with agents on one side having their individual preferences for agents on another side as a potential match. It is assumed silently that agents can freely and costlessly ``switch" partners. A matching is called stable if no two unmatched agents prefer each other to their matches. Half a century ago, for equinumerous sides, Knuth demonstrated existence of preferences for which there are exponentially many stable matchings, and he posed a problem of evaluating an expected number of stable matchings when the preferences are uniformly random and independent. It was shown later by Pittel that this expectation is quite moderate, asymptotic to $e^{-1}n\log n$, $n$ being the number of agents on each side. The proof used Knuth's integral formula for the expectation based on a classic inclusion-exclusion counting. In later papers by Pittel this and other integral formulas were obtained by generating preference lists via a pair of random matrices, with independently random, $[0,1]$-uniform entries, rather than from a progressively problematic counting approach. The novelty of this paper is that we view the matrix entries as a basis for cardinal utilities that each agent ascribes to the agents on the other side. Relaxing the notion of stability, we declare matching stable if no unmatched pair of agents {\it strongly\/} prefer each other to their partners, with ``strength'' measured by a parameter $\eps>0$, $\eps=0$ corresponding to classic stability. We show that 
for $\eps$ of order $n^{-1}\log n$ the expected number of $\eps$-stable matchings is polynomially large, but 
for $\eps\gg n^{-1}\log n$ {\it however slightly\/} the expectation is suddenly super-polynomially large. This ``explosion'' phenomenon for a very small strength parameter $\eps$ continues to hold for imbalanced matchings, regardless of the imbalance size.
\end{abstract}

\section{Introduction and main results}

In a standard stable matching problem, we have two finite sets: 
$n$ ``firms'' $\{f_i\}_{i\in [n]}$ and $n+k$ ``workers" $\{w_j\}_{j\in [n+k]}$, $k\ge 0$ measuring the imbalance of the two-sided
``market''. Each agent ranks the agents on the other side from most preferred to least preferred, with no ties. A matching is a one-to-one assignment (matching) $M$ of $n$ workers to $n$ firms. The classic definition of stability is: (a) $M$ is stable
if no two agents, both matched in $M$ but with somebody else, prefer each other to their partners in $M$; (b) each agent matched in $M$ prefers its partner in $M$ to any other unmatched agent on the other side.

Following the tradition, we consider the case when each worker's (firm's)  preference list is chosen uniformly at random, and independently from all $n!$ orderings of firms, (from all $(n+k)!$ orderings of workers respectively). Admittedly, it is a highly idealized model of the random preferences, its huge advantage being a possibility to learn intricate details of such a random market, some of which hopefully will survive in a broader class of random preferences, when a rigorous analysis is
out of reach entirely. 

The reason for this model's surprising tractability, and its usefulness for our goal to develop a relaxed stability notion, is existence of a conceptually simple way to generate the uniform/independent random preferences. Here it is. Introduce two
$n\times (n+k)$ matrices 
\[
\bold X=\{X_{i,j}\}_{i\in [n],j\in [n+k]},\quad \bold Y=\{Y_{i,j}\}_{i\in [n],j\in [n+k]},
\]
whose $2n(n+k)$ entries are independent random variables, each distributed uniformly on $[0,1]$. We postulate that each 
firm $f_i$ (worker $w_j$ resp.) ranks the $(n+k)$ workers (the $n$ firms, resp.) in the increasing order of row entries $X_{i,j}$(column entries $Y_{i,j}$ resp.). Continuity of the uniform distribution implies that with probability $1$ there are no ties. Clearly, the resulting preference lists are all uniformly random and mutually independent. The appearance of $\bold X,\bold Y$ is unavoidable.
The $(\bold X ,\bold Y)$--induced preference model is {\it the limit of a discrete model\/}, when each agent's ranking of potential partners is obtained by a uniformly random allocation of their names among the ordered set of $n$ ($n+k$ resp.)``boxes'', which is then conditioned on the event ``no two or more balls in the same box''.

As a warm-up, let us see how this ``embedding'' works for the standard stability. Consider a ``diagonal'' matching $M$ that matches firm $f_i$ with worker $w_i$, $i\in [n]$. Introduce the events $A_{i,j}$, $i\in [n]$, $j\in [n+k]$, $i\neq j$:
\[
A_{i,j}=\left\{\begin{aligned}
&(X_{i,j}<X_{i,i},\, Y_{i,j}<Y_{j,j},)&&i,j\in [n],\\
&(X_{i,j}<X_{i,i}),&&i\in [n], j\in [n+1,n+k].\\
\end{aligned}\right.
\]
Clearly $\{M\text{ is stable}\}=\bigcap_{i,j}A^c_{i,j}$.
The interdependent events $A_{i,j}$ become independent, when we condition them on the event ``$X_{i,i}=x_i$, $Y_{j,j}=y_j$, $i,j\in [n]$''. Denoting the conditioning in question by $|\circ)$, we have
\[
\Prob(A_{i,j}|\circ)=\left\{\begin{aligned}
&\Prob\bigl(X_{i,j}\le x_i,\,Y_{i,j}<y_j\bigr)=x_iy_j, &&i,j\in [n],\\
&\Prob\bigl(X_{i,j}\le x_i\bigr)=x_i,&&i\in [n], j\in [n+1,n+k].
\end{aligned}\right.
\]
Integrating over $x_i, y_j\in [0,1]$, $i,j\in [n]$, we obtain
\begin{multline}\label{0.9}
P(n,n+k):=\Prob(M\text{ is stable})\\
=\int\limits_{\bold x,\bold y\in [0,1]^n}\prod_{1\le i\neq j\le n}(1-x_iy_j)\cdot \prod_{\ell\in [n]}
(1-x_{\ell})^k\,d\bold x d\bold y.
\end{multline}
(It is by no means clear at the moment how to handle this, and similar integrals below, but obviously the enumerative issue is not a problem any more.) For $k=0$,  a  ``balanced market'' in the economics language, we get Don Knuth's formula for $P(n)=P(n,n)$, the probability that the diagonal matching $M=\{(i,i)\}$ is stable. 
Knuth \cite{Knu} 
discovered this striking formula by applying a combinatorial inclusion-exclusion identity, identifying each sign-alternating term as an integral over $[0,1]^{2n}$, interchanging summation and integration, and observing that the resulting integrand equals 
$\prod_{1\le i\neq j\le n}(1-x_iy_j)$. 

The above argument for $k=0$ was found in \cite{Pit1}, and the technique was used  in a series of other papers by Pittel (like the formula for $P(n,n+k)$ in \cite{Pit2}) in the cases when Knuth's ingenious approach would have been progressively more problematic, if possible at all. 

By analyzing the asymptotic behavior of the integral representing $P(n)$ it was proved in \cite{Pit1} that the expected number of stable matchings in the balanced case is asymptotic to $e^{-1}n\log n$, thus growing a bit faster than a linear function, in a stark contrast with Knuth's discovery of deterministic preference lists resulting in exponential many stable matchings. This certainly signaled that the classic definition of stability was unrealistically rigid, giving rise to a problem: how can we relax definition of stability which would reflect better the reality of decision--making in the two-side markets when changing partners is far from being costless, whatever the nature of these costs is?

While the entries of those matrices $\bold X, \bold Y$ can certainly be interpreted as utility measures of potential partners,
in Pittel's work they were used only as instruments for computing/estimating expected numbers of various stable matchings. 
In this work we decided to actually use  $\{X_{i,j}, Y_{i,j}\}$ as the basis for utility measures, thus to view those entries as payoff parameters in a game-theoretical model with many agents. (Among our sources of inspiration we should mention,
for instance, McLennan \cite{McL}, McLennan and Berg \cite{McB}.) 

Our strategy was to use an auxiliary parameter $\eps>0$ as a threshold for a potential increase of utility making it feasible for an agent to switch to another partner, who is similarly encouraged to break up with his/her partner. The question is: how do we measure utility for an agent associated with a potential partner? We postulate existence of a pair of {\it strictly decreasing} functions $f(z)$ and $w(z)$ such that utility
for firm $i\in [n]$ associated with worker $j\in [n+k]$ is $f(X_{i,j})$, and similarly $w(Y_{i,j})$ is utility for worker $j\in [n+k]$ associated with firm $i\in [n]$. Thus, the closer a worker (firm) is to the top of preference list for a firm, the larger the utility of the worker (firm) is for the firm (worker). Given $f(z)$, $w(z)$ {\it and} $\eps>0$, we declare a matching $M$ between $n$ firms and some $n$ workers $
\eps$-stable if:  ${\bf (a)\/}$ no pair $(\text{firm},\text{worker})=(i, j)$, such that worker $j$ is matched, $\eps$-prefer each other to their partners under matching $M$ meaning that
\begin{equation}\label{1}
f(X_{i,j})>f(X_{i,M(i)})+\eps,\quad w(Y_{i,j})>w(M(j),j)+\eps;
\end{equation}
 ${\bf (b)\/}$ for no pair $(i,j)$ such that $j$ is not matched, $i$ $\eps$-prefers $j$ to its partner in $M$,
 meaning that
\begin{equation}\label{2}
f(X_{i,j})>f(X_{i,M(i)})+\eps.
\end{equation}
{\it Clearly, the case $\eps=0$ gets us back to the standard stability.\/} Not surprisingly, chances for getting inside into issue of how $\eps$ influences the total number of resulting $\eps$-stable matchings strongly depend on the choice of utility functions $f$ and 
$w$. Initially, decreasing linear functions $f$ and $w$ looked appealing, but then it occurred to us that it is the {\it ratio\/} $X_{i,j}/X_{i,M(i)}$
($Y_{i,j}/Y_{M(j),j}$), which is, intuitively, a better measure of firm $i$'s (worker $j$'s) willingness to switch to worker $j$ (firm $i$)
at the expense of worker $M(i)$ (firm $M(j)$). 

Indeed, these ratios can be viewed as the {\it limiting\/} ratio of the respective ranks of the worker $j$ and the matching worker $M(i)$, and the {\it limiting\/} ratio of the respective ranks of the firm $i$ and the matching firm $M(j)$ in the firm $i$'s 
(the worker $j$'s) preference list. If so, we have almost no choice but {\it define\/} $f(z)=\frac{1}{\la'}\log(1/z)$, $w(z)=\frac{1}{\la''}\log(1/z)$. For simplicity, we assume that $\la'=\la''=\la$. 

What is a counterpart of the formula \eqref{0.9} for $P^{\eps}(n,n+k)$, the probability that the diagonal matching $M=\{i,i\}_{i\in [n]} $ is $\eps$-stable? Introduce the $\eps$-relaxed events $B_{i,j}$:
\[
B_{i,j}=\left\{\begin{aligned}
&\left(\frac{1}{\la}\log\frac{X_{i,M(i)}}{X_{i,j}}\ge\eps;\,
\frac{1}{\la}\log\frac{Y_{M(j),j}}{Y_{i,j}}\ge\eps\right),\,\,i,j\in [n],\\
&\left(\frac{1}{\la}\log\frac{X_{i,M(i)}}{X_{i,j}}\ge\eps\right),\,\,i\in [n],\, j\in [n+1,n+k].\\
\end{aligned}\right.
\]
This time $(M\text{ is }\eps-\text{stable})=\bigcap_{i,j}B_{i,j}^c$, and the events $B_{i,j}$ are independent, when conditioned on $\{X_{i,i}=x_i,
Y_{i,j}=y_j\}_{i,j\in [n]}$. By the definition of $B_{i,j}$,
\[
\Prob(B_{i,j}^c|\circ)=\left\{\begin{aligned}
&1-\Prob(B_{i,j}|\circ)=1-e^{-2\eps\la}x_iy_j,&&i,j\in [n],\\
&1-\Prob(B_{i,j}|\circ)=1-e^{-\eps\la} x_i, &&i\in [n], j\in [n+1,n+k].\end{aligned}\right.
\]
Therefore
\begin{equation}\label{3}
P^{\eps}(n,n+k)=\int\limits_{\bold x,\bold y\in [0,1]^n}\!\prod_{i\neq j}(1-e^{-2\eps\lambda}x_iy_j)\cdot\prod_{\ell\in [n]}(1-e^{-\eps\lambda}x_{\ell})^k\,d\bold x d\bold y.
\end{equation}
By symmetry, $S^{\eps}(n,n+k)$, the expected total number of $\eps$-stable matchings is given by $S^{\eps}(n,n+k)=(n+k)_n
P^{\eps}(n,n+k)$. ($(n+k)_n$ is the total number of ways to pick up, with order, $n$ workers out of $(n+k)$ workers.) So, an asymptotic analysis of $P^{\eps}(n,n+k)$ hopefully may reveal dependence of $S^{\eps}(n,n+k)$ on $\eps>0$. How large should $\eps$ be to result in a truly substantial increase of that expectation? Will this increase be dramatically
``sudden'', like a phase transition in statistical physics models?

{\bf Note.\/} For $k=0$, the integrand reduces to $\prod_{i\neq j}(1-e^{-2\eps\lambda}x_iy_j)$.  A similar integrand,
namely, $p^n\prod_{i\neq j}(1-px_iy_j)$ appeared in \cite{Pit3}. The associated integral was the probability that diagonal stable matching is stable in the usual sense for the case of constrained matching, with complete preference lists, subject to admissibility test for each pair of partners. The test results are independent, and $p$ is the probability that a given pair passes the test. It was proved in \cite{Pit3} that $p=
n^{-1}\log ^2n$ is a threshold probability for existence of a stable matchings. It is this experience that made us hopeful about our chances to get definitive results for our model as well. 

Here is our central result.
\begin{theorem}\label{thm1} {\bf(i)\/} If $\eps\la=O(n^{-1}\log n)$, then $S^{\eps}(n,n+k)$ is polynomially large, uniformly over $k\ge 0$. {\bf(ii)\/} $\eps\la=\omega(n^{-1}\log n)$, i. e. $\eps\la\gg n^{-1}\log n$, $S^{\eps}(n,n+k)$ is super-polynomially large, again uniformly over $k\ge 0$. 
\end{theorem}
Thus, there is indeed a sudden ``explosion" of the expected number of $\eps$-stable matching when $\eps\la$ switches from a bounded multiple of $n^{-1}\log n$ to $\omega n^{-1}\log n$ where $\omega\to\infty$ however slowly. Strikingly, 
$n^{-1}\log n$ goes to zero very fast as $n$, the market size, goes to infinity. That is, the transition takes place when the stability model is very close to the standard one. Importantly, this phenomenon holds both for $k=0$ (balanced market) and $k>0$ (imbalanced market), though a breakthrough paper by Ashlagi, Kanaria, and Leshno \cite{AKL}, and \cite{Pit2} demonstrated a dramatic reduction of the number of standard stable matchings for the imbalanced markets, even for the minimal imbalance $k=1$.

We emphasize that it was the assumption that the random utilities are logarithmic functions of the $[0,1]$-Uniforms, thus exponentially distributed random variables, that rendered our model of stable matchings with switching costs  amenable to analysis. To the best of our knowledge, no general structural results are known for $\eps$-stable matchings, except that the standard stable matchings are $\eps$--stable. 

In \cite{Pit1}, \cite{Pit1.5}, \cite{Pit2}, \cite{Pit3}, and other papers by Pittel, the asymptotic analysis of the standard stable matchings went beyond 
the expected number of stable matchings, covering also the limiting behavior of the ranks of stable, including a ``law of hyperbola'', stating (for $k=0$) that for each of stable matchings the product of total ranks of stable partners
is likely to be close to $n^3$. Our theorem may serve as an encouragement to search for extensions of some of those results to $\eps$-stable matchings in the post-critical phase $\eps\la>\!>n^{-1}\log n$. 

\section{Proofs} We consider the cases $k=0$ and $k>0$ separately. In each case, the analysis is based on the properties of a random partition of $[0,1]$. Let $X_1,\dots, X_{\ell}$ denote the independent random variables, each distributed uniformly on $[0,1]$. Set $\mathcal S_{\ell}=\sum_{j\in [\ell]}X_j$, $V_j=\tfrac{X_j}{\mathcal S_{\ell}},\, (j\in [\ell])$, so that $V_j\in [0,1]$, $\sum_{j\in [\ell]}V_j=1$. Introduce $f_{\ell}(s,v_1,\dots,v_{\ell-1})$, the density of 
$(\mathcal S_{\ell},V_1,\dots V_{\ell-1})$, which is non-zero for $v_j\le \tfrac{1}{s}$, $\sum_{j<\ell}v_j\le 1$. The Jacobian of $\bold x=(x_1,\dots, x_{\ell})$ with respect to $(s,v_1,\dots, v_{\ell-1})$ is $s^{\ell-1}$, implying that
\begin{equation}\label{A}
 f_{\ell}(s,v_1,\dots,v_{\ell-1})=s^{\ell-1}\,\Bbb I\Bigl(\sum_{j<\ell}v_j\le 1\text{ and } \max_{j\le \ell}v_j\le\tfrac{1}{s}\Bigr),
\end{equation}
where $v_{\ell}:=1-\sum_{j<\ell} v_j$. Now, $(\ell-1)!\cdot\Bbb I\Bigl(\sum_{j<\ell}v_j\le 1\Bigr)$ is the joint density of $L_1,\dots,L_{\ell-1}$,  the lengths of the first $(\ell-1)$ consecutive intervals obtained by selecting  $(\ell-1)$ points 
independently and uniformly at random in $[0,1]$. So, integrating \eqref{A} over $v_1,\dots, v_{\ell-1}$, we obtain: 
the density of $S_{\ell}$ is given by
\begin{equation}\label{B}
f_{\ell}(s)=\tfrac{s^{\ell-1}}{(\ell-1)!}\cdot \Bbb \Prob \bigl(\max_{j\in [\ell]}L_j\le \tfrac{1}{s}\bigr), 
\end{equation}
cf. Feller \cite{Fel}. Dropping the constraint  $\max_{j\le \ell}v_j\le\tfrac{1}{s}$ in \eqref{A}, we obtain a crucial 
inequality
\begin{equation}\label{C}
 f_{\ell}(s,v_1,\dots,v_{\ell-1})\le \tfrac{s^{\ell-1}}{(\ell-1)!}\cdot g_{\ell}(v_1,\dots, v_{\ell-1}),
 \end{equation}
 where $g_{\ell}(v_1,\dots, v_{\ell-1})$ is the density of $L_1,\dots, L_{\ell-1}$.
  
To use this connection to $L_1,\dots, L_{\ell}$, we will need asymptotic properties of $\L_{\ell}^+=\max_{j\in [\ell]}L_j$ and $U_{\ell}:=\sum_{j\in [\ell]}L_j^2$:

\begin{lemma}\label{lem0} As $\ell\to\infty$, in probability $\tfrac{L_{\ell}^+}{\ell^{-1}\log \ell}\to 1$, and $\ell\, U_{\ell}\to 2$. More sharply, for every $\rho>0$, we have 
\[
\Bbb P\bigl(L_{\ell}^+\le \ell^{-1}(\log\tfrac{\ell}{\log \ell}-\rho)\bigr)=O(\ell^{-d}), \quad
\forall\,d\in (0,e^{\rho}-1).,
\]
and for $\omega(\ell)\to\infty$ however slowly 
\[
\Bbb P\bigl(L_{\ell}^+\le \ell^{-1}\log(\ell\omega(\ell))\bigr)=(1+o(1))\exp\bigl(-\tfrac{1+o(1)}{\omega(\ell)}\bigr)\to 1, 
\]
while, for every $\delta<\tfrac{1}{3}$,
\[
\Bbb P\biggl(\big|\tfrac{\ell\, U_{\ell}}{2}-1\big|\ge \ell^{-\delta}\biggr)=O\bigl(e^{-\Theta(\ell^{1/3-\delta})}\bigr).
\]
\end{lemma}
\noindent {\bf Note.\/} 
The proof of  Lemma \ref{lem2}, is based on the following  fact: $\{L_j: j\in [\ell]\}$ has the same distribution as $\bigl\{W_j/\sum_{i\in [\ell]} W_i: j\in [\ell]\bigr\}$, where $W_1,\dots,W_{\ell}$ are i.i.d. exponentials, with the common parameter $1$, say. The large deviation estimate for $\ell\, U_{\ell}$ was proved by Pittel \cite{Pit3}.
\subsection{The case $k=0$.}

\begin{lemma}\label{lem1} If $\eps \lambda =O(n^{-1}\log n)$, then $\E[S^{\eps}(n,n)]$ is polynomially large.
\end{lemma}
\begin{proof} As shown before,  $\E[S^{\eps}(n,n)]=n! P^{\eps}(n)$, where
\begin{equation*}
    P^{\eps}(n) = \int\limits_{\bold x,\bold y\in [0,1]^n}\!\prod_{i\neq j}(1-px_iy_j)\,d\bold x d\bold y \quad \text{and} \quad p \equiv \exp(-2\eps\lambda).
\end{equation*}
We prove the desired result by analyzing an upper bound for $P^{\eps}(n)$. 


Introduce $s\equiv \sum_{i\in [n]}x_i$ and $s_j \equiv \sum_{i\neq j}x_i$. Since $1-\xi \leq \exp(-\xi)$, we get
\begin{equation}\label{9}
P^{\eps}(n)\leq \int\limits_{\bold x\in [0,1]^n}\left(\prod_{j\in [n]}\int_0^1\exp\left(-pys_j\right)\,dy\right)
d\bold x,
\end{equation}
Integrating with respect to $y\in [0,1]$, we obtain 
\begin{equation}\label{9.1}
 P^{\eps}(n) \leq \int\limits_{\bold x\in [0,1]^n}\prod_{j\in [n]}F(s_j p)\,d\bold x, \quad \text{where} \quad F(z) \equiv \tfrac{1-e^{-z}}{z}.   
\end{equation}

To simplify the above estimate, note that for  $z>0$,
\begin{equation*}
0>\left(\log F(z)\right)'=\tfrac{1}{e^z-1}-\tfrac{1}{z}\begin{cases}
\to -\tfrac{1}{2} \quad &\text{as} \quad z\downarrow 0,\\
\sim -\tfrac{1}{z} \quad &\text{as} \quad z\uparrow\infty.
\end{cases}
\end{equation*}
By defining $A \equiv \sup_{z>0}(z+1)\left|\left(\log F(z)\right)'\right|>0$, we have
\begin{equation*}
\Big|(\log F(z))'\Bigr|\le \tfrac{A}{z+1}.
\end{equation*}
Therefore, since $s_j=s-x_j$,
\begin{align*}
\log F(s_jp)
&=\log F(sp)-\int\limits_{s-x_j}^s(\log F(zp))'_z\,dz \leq \log F(sp)+\int\limits_{s-x_j}^s p\, \frac{A}{pz+1}\,dz\\
&\leq \log F(sp)+\int\limits_{s-x_j}^s p\, \frac{A}{p(s-x_j)+1}\,dz\\
& \leq \log F(sp)+\frac{Apx_j}{(s-x_j)p+1}\le \log F(sp)+\frac{Ax_j}{s}.
\end{align*}
Thus, by the definition of $s$,
\begin{equation*}
\prod_{j\in [n]}F(s_jp)\le e^A F^n(sp).    
\end{equation*}

Then, by changing variables and using \eqref{9.1},
\begin{align*}
P^{\eps}(n) &\leq \int\limits_{\bold x\in [0,1]^n} e^A F^n(sp)\,d\bold x \leq e^{A}\int_0^{n}F^n(sp)\,\frac{s^{n-1}}{(n-1)!}\,ds\\
&= e^{A}\int_0^{n}\biggl(\frac{1-e^{-sp}}{sp}\biggr)^n\frac{s^{n-1}}{(n-1)!}\,ds=\int_0^{n}\frac{e^Ap^{-n}}{(n-1)!}\frac{(1-e^{-sp})^n}{s}\,ds.
\end{align*}


We fix $s(n) \equiv \tfrac{\log n}{2p}$ and break the last integral into two parts, $\int_1$ for $s \leq s(n)$ and $\int_2$ for $s > s(n)\}$. In what follows, we provide upper bounds for these two integrals.\bigskip

\noindent \textit{Integral 1.}
Note that
\begin{align*}
\int_1=\frac{e^Ap^{-n}}{(n-1)!}\int_0^{ps(n)}\exp(\Phi_n(\eta))\,d \eta
\end{align*}
where $\Phi_n(\eta) \equiv n\log(1-e^{-\eta})-\log\eta$. 
It is straightforward to verify that $\Phi_n(\eta)$ is increasing and concave on $[0,ps(n)]$ for sufficiently large $n$. Indeed, since $\tfrac{e^{\eta}-1}{\eta}$ is increasing for $\eta>0$,
\begin{align*}
\Phi'_n(\eta)&=\frac{n}{e^{\eta}-1}-\frac{1}{\eta}=\frac{n}{e^{\eta}-1}\bigl(1-\frac{e^{\eta}-1}{n\eta}\bigr)\\
&\geq \frac{n}{e^{\eta}-1}\left(1-\frac{e^{ps(n)}-1}{nps(n)}\right)=\frac{n}{e^{\eta}-1}(1-O(n^{-1/2}))>0,
\end{align*}
if $n$ is sufficiently large, where the last equality follows from $ps(n)=\tfrac{\log n}{2}$. Similarly, since $\tfrac{e^{\eta/2}-e^{-\eta/2}}{\eta}$ is increasing for $\eta>0$, 
\begin{align*}
\Phi_n''(\eta)&=-\frac{n}{(e^{\eta/2}-e^{-\eta/2})^2}+\frac{1}{\eta^2}=-\frac{n}{(e^{\eta/2}-e^{-\eta/2})^2}\left(1-\frac{(e^{\eta/2}-e^{-\eta/2})^2}{n\eta^2}\right)\\
&\leq-\frac{n}{(e^{\eta/2}-e^{-\eta/2})^2}\left(1-\frac{\bigl(e^{ps(n)}-e^{-ps(n)}\bigr)^2}{n(ps(n))^2}\right)\\
&=-\frac{n}{(e^{\eta/2}-e^{-\eta/2})^2}(1-O(\log^{-2}n)) < 0,
\end{align*}
for sufficiently large $n$.

As a result, by using $\Phi_n(\eta)\le \Phi_n(ps(n))+\Phi_n'(ps(n))(\eta-ps(n))$ for $\eta\in [0,ps(n)]$, we obtain

\begin{align*}
\int_1 = \frac{e^A p^{-n}}{(n-1)!}\int_0^{ps(n)}\!\!\!e^{\Phi_n(\eta)}\,d \eta&\le\frac{e^{A}p^{-n}}{(n-1)!} e^{\Phi_n(ps(n))}
\int_0^{ps(n)}\!\!\!\!\exp\bigl(\Phi_n'(ps(n))(\eta-ps(n))\bigr)\,d\eta\\
&\leq \frac{e^{A}p^{-n}}{(n-1)!}\cdot\frac{e^{\Phi_n(ps(n))}}{\Phi_n'(ps(n))}.
\end{align*}
To sum up, the contribution of $n! \int_1$ to $\E[S^{\eps}(n,n)]$ is
\begin{equation}\label{a1}
\begin{aligned}
    n! \int_1 
    &\leq n!\, \tfrac{e^{A}p^{-n}}{(n-1)!}\cdot\frac{(1-e^{-ps(n)})^n}{ps(n)}\cdot \frac{e^{ps(n)}-1}{n} \left(1-\frac{e^{ps(n)}-1}{nps(n)}\right)^{-1}\\
    & = O(p^{-n}n^{1/2})=\exp(O(\log n)),
\end{aligned}
\end{equation}
since $p=e^{-2\eps \lambda}$ and $\eps\lambda=O(n^{-1}\log n)$.\bigskip

\noindent \textit{Integral 2.} We now turn attention to the second integral
\begin{equation*}
 \int_2 = \int_{s(n)}^{n}\frac{e^Ap^{-n}}{(n-1)!}\cdot\frac{(1-e^{-sp})^n}{s}\,ds
\leq \frac{e^Ap^{-n}}{(n-1)!}\cdot \int_{s(n)}^{n}\tfrac{1}{s}\,ds \leq \frac{e^Ap^{-n}}{(n-1)!} \log n.
\end{equation*}
Consequently, the contribution of $n! \int_2$ to $\E[S^{\eps}(n,n)]$ is 
\begin{align*}
 n! \int_2 = O \left(n p^{-n} \log n \right)=\exp(O(\log n)),
\end{align*}
because $p=e^{-2\eps \lambda}$ and $\eps\lambda=O(n^{-1}\log n)$. Combining this bound with \eqref{a1}, we complete the proof of Lemma \ref{lem1}.
\end{proof}

\begin{lemma}\label{lem2} If $\eps \lambda =\omega(n^{-1}\log n)$, i.e. $\eps \lambda\gg n^{-1}\log n$, then $\E[S^{\eps}(n,n)]$ is super-polynomially large.
\end{lemma}

\begin{proof}
Since $\E[S^{\eps}(n,n)]$ is increasing in $\epsilon\lambda$, and $\epsilon\lambda=\omega(n^{-1}\log n)$, it is sufficient to focus only on $p=e^{-2\epsilon\lambda}=1-o(1)$ in our calculations below.

Define the following variables:
\begin{equation*}
    s \equiv \sum_{i \in [n]}x_i,\quad s_j \equiv \sum_{i\neq j}x_i=s-x_j, \quad t \equiv \sum_{i \in [n]}x_i^2 , \quad \text{and} \quad t_j\equiv \sum_{i\neq j}x_i^2=t-x_j^2.
\end{equation*}
Restrict attention to the set $D$ of all non-negative $x=(x_1,x_2,\ldots,x_n)$ such that
\begin{align}
\label{set_D}
\begin{split}
&\frac{n}{\log n\cdot\log(ns_1(n))}  \leq s \leq \frac{n}{\log n\cdot\log(ns_2(n))}, \\
&s^{-1} x_i \leq \frac{\log(ns_2(n))}{n},\quad i \in [n],\\
&s^{-2} t \leq \frac{3}{n},
\end{split}
\end{align}
where $s_1(n)>s_2(n)$, with $s_2(n) \to \infty$ as $n \to \infty$, are to be defined later. (The two bottom inequalities are essentially dictated by the connection between the fractions $s^{-1}x_i$ and the lengths $L_i$ of random subintervals 
of $[0,1]$, and Lemma \ref{lem0}.) According to the first two conditions, we have
\begin{equation*}
    x_i \leq s \frac{\log(ns_2(n))}{n} = (\log n)^{-1}<1
\end{equation*}
for sufficiently large $n$, so that $D \subset [0,1]^n$. Consequently, for such large $n$,
\begin{equation*}
P^\epsilon(n)\geq \tilde{P}^\epsilon(n) \equiv  \int_{\bold x \in D}\!\left(\prod_{j \in [n]}\left(\int_0^1 \prod_{i \neq j} (1-px_iy_j)\,dy_j\right)\right)\,d\bold x.
\end{equation*}

Since $1-\alpha=\exp(-\alpha-\alpha^2(1+O(\alpha))/2)$ as $\alpha \to 0$, for each fixed $j$, we have
\begin{align*}
    \prod_{i \in[n], i\neq j} \left(1-px_iy_j\right) 
    &= \exp\left(-py_js_j-p^2y^2_jt_j\frac{1+O\left((\log n)^{-1}\right)}{2}\right)\\
    &\geq \exp\left(-py_js_j-\frac{2p^2y^2_jt_j}{3}\right)\\
    &\geq \exp\left(-py_js-\frac{2p^2y^2_jt}{3}\right).
\end{align*}

Therefore,
\begin{align*}
\prod_{j \in [n]}\left(\int_0^1 \prod_{i \neq j} (1-px_iy_j)\,dy_j\right)
&\geq \left(\int_0^1 \exp\left(-pys-\frac{2p^2y^2t}{3}\right)\,dy\right)^n\\
&= (ps)^{-n} \left( \int_0^{ps} \exp \left( -\eta - \frac{2\eta^2 t}{3 s^2} \right) d\eta \right)^n \\
&\geq (ps)^{-n} \left( \int_0^{ps} \exp \left( -\eta - \frac{2 \eta^2}{n} \right) d\eta \right)^n
\end{align*}
since $t / s^2 \leq \frac{3}{n}$ for $\bold x \in D$.

By noting that $\int_{0}^{ps} e^{-\eta} d\eta = 1-e^{-ps}$ and employing Jensen's inequality for the exponential function, we can bound
\begin{align*}
    \int_0^{ps} \exp \left( -\eta - \frac{2 \eta^2}{n} \right) d\eta 
    &= (1-e^{-ps})\int_0^{ps} \exp \left(- \frac{2 \eta^2}{n} \right) \cdot \frac{e^{-\eta}}{1-e^{-ps}} d\eta\\
    &\geq(1 - e^{-ps}) \exp \left( -\frac{2 / n}{1 - e^{-ps}} \int_0^{ps} e^{-\eta} \eta^2 d\eta \right) \\
    &= (1 - e^{-ps}) \exp \left( -\frac{4}{n} \cdot \frac{ e^{ps} - \left(1 + ps + (ps)^2 / 2\right)}{e^{ps} - 1} \right).
\end{align*}
Since
\begin{equation*}
\frac{ e^{ps} - \left(1 + ps + (ps)^2 / 2\right)}{e^{ps} - 1} \leq 1,
\end{equation*}
we obtain
\begin{equation*}
\int_0^{ps} \exp \left( -\eta - \frac{2 \eta^2}{n} \right) d\eta \geq (1 - e^{-ps}) e^{-4 / n}.   
\end{equation*}

Consequently,

\begin{equation*}
\prod_{j \in [n]} \left( \int_0^1 \prod_{i \neq j} (1 - p x_i y_j) dy_j \right) \\
\geq e^{-4} \left(\frac{1 - e^{-ps}}{ps} \right)^n,  
\end{equation*}
implying that
\begin{equation*}
\tilde{P}^\epsilon(n) \geq e^{-4} p^{-n} \int_{\bold x\in D}\! \left(\frac{1 - e^{-ps}}{s} \right)^n\,d\bold x.
\end{equation*}

We then switch to new variables $(u,v_1,v_2,\ldots,v_{n-1})$ defined as follows:
\begin{equation*}
    u \equiv \sum_{i \in [n]} x_i = s \qquad \text{and} \qquad v_i \equiv
    s^{-1}x_i, \quad i \in [n-1].
\end{equation*}
We also define $v_n \equiv s^{-1}x_n=1-\sum_{i \in [n-1]}v_i$. The conditions \eqref{set_D} then become
\begin{align}
\begin{split}
&\frac{n}{\log n\cdot\log(ns_1(n))}  \leq u \leq \frac{n}{\log n\cdot\log(ns_2(n))}, \\
&v_i \leq \frac{\log(ns_2(n))}{n},\quad i \in [n],\\
&\sum_{i \in [n]}v_i^2\leq \frac{3}{n},
\end{split}
\end{align}
Importantly, for these new variables, the resulting region is the direct product of the range of $u$ and the range of $\bold v$. By using the key equality \eqref{A},  we get then
\begin{align*}
\tilde{P}^\epsilon(n)& \geq \frac{e^{-4}p^{-n}}{(n-1)!} \left( \int\limits_{\frac{n}{\log n\cdot\log(n s_1(n))}}^{\frac{n}{\log n\cdot\log(ns_2(n))}} \frac{(1 - e^{-pu})^n}{u} du \right)\\
&\quad\times \Prob\left(L^+_n  \leq \frac{\log(n s_2(n))}{n} \quad \text{and} \quad U_n \leq \frac{3}{n}\right),
\end{align*}
where $L^+_n \equiv  \max_{i \in [n]} L_i$ and $U_n \equiv \sum_{i \in [n]}L_i^2$. By Lemma \ref{lem0},  the probability of the event on the right-hand side tends to one as $n \to \infty$. Choose $s_1(n)=o(\log n)$ and $s_2(n)=\frac{s_1(n)}{2}$. Then,

\begin{align*}
\int\limits_{\frac{n}{\log n\cdot\log(n s_1(n))}}^{\frac{n}{\log n\cdot\log(ns_2(n))}} \frac{1}{u} du =\log \left(\frac{\log(ns_1(n))}{\log(ns_2(n))}\right)
\end{align*}
is asymptotic to $\frac{\log 2}{\log n}$. Furthermore,  since $p = 1-o(1)$, for the considered range of $u$, we have
\begin{align*}
    pu \geq \frac{pn}{\log n\cdot\log(n s_1(n))} \geq \frac{n}{2(\log n)^2},
\end{align*}
so that
\begin{align*}
    (1 - e^{-pu})^n = \left(1-\exp\left(-\frac{n}{2(\log n)^2}\right)\right)^n = 1-o(1).
\end{align*}

To conclude,
\begin{align*}
\tilde{P}^\epsilon(n) \geq \frac{e^{-4}p^{-n}}{(n-1)!}  \frac{\log 2}{\log n} (1-o(1)),  
\end{align*}
and thus
\begin{align*}
\E[S^\epsilon (n,n)] \geq n! \tilde{P}^\epsilon(n)  \geq \frac{n e^{-4}p^{-n}\log 2}{\log n} (1-o(1))=\exp(\omega(\log n))
\end{align*}
since $p=e^{-2\eps \lambda}$ and $\eps\lambda=\omega(n^{-1}\log n)$.
\end{proof}

Combining Lemma \ref{lem1} and Lemma \ref{lem2} we complete the proof of Theorem \ref{thm1} for $k=0$.
\subsection{The case $k\ge 1$.}
\begin{lemma}\label{lem3}  If $\eps \lambda =O(n^{-1}\log n)$, then $\E[S^{\eps}(n,n+k)]$ is polynomially large. 
\end{lemma}
\begin{proof} As shown before,  $\E[S^{\eps}(n,n+k)]=(n+k)_n P^{\eps}(n,n+k)$, where
\begin{equation*}
P^{\eps}(n,n+k)=\int\limits_{\bold x,\bold y\in [0,1]^n}\!\prod_{i\neq j}(1-px_iy_j)\cdot\prod_{\ell\in [n]}(1-\sqrt{p}x_{\ell})^k\,d\bold x d\bold y \quad \text{and},
\end{equation*}
$p =\exp(-2\eps\lambda)$. Let us  upper bound $P^{\eps}(n,n+k)$. 
Since $1-\xi \leq \exp(-\xi)$, we get
\begin{equation}\label{10}
P^{\eps}(n,n+k)\leq \int\limits_{\bold x\in [0,1]^n}\left(\prod_{j\in [n]}\int_0^1\exp\left(-pys_j\right)\,dy\right)\exp\left(-k\sqrt{p}s
\right)
d\bold x,
\end{equation}
where
\begin{equation*}
    s \equiv \sum_{i \in [n]}x_i \quad \text{and} \quad s_j \equiv \sum_{i\neq j}x_i=s-x_j.
\end{equation*}
Integrating with respect to $y\in [0,1]$, we obtain 
\begin{equation}\label{a2}
 P^{\eps}(n,n+k) \leq \int\limits_{\bold x \in [0,1]^n}\left(\prod_{j\in [n]}F(s_j p)\right)\exp\left(-k\sqrt{p}s\right)\,d\bold x, \quad F(z) =\frac{1-e^{-z}}{z}.   
\end{equation}
As in the proof of Lemma~\ref{lem1}, for some $A>0$, we have
\begin{equation*}
\prod_{j\in [n]}F(s_jp)\le e^A F^n(sp).    
\end{equation*}
Then, by changing variables and using \eqref{a2}, we obtain
\begin{align*}
&P^{\eps}(n,n+k) \leq \int\limits_{\bold x \in [0,1]^n} e^A F^n(sp)e^{-k\sqrt{p}s}\,d\bold x \leq e^{A}\int_0^{n}F^n(sp)e^{-k\sqrt{p}s}\,\frac{s^{n-1}}{(n-1)!}\,ds\\
&= e^{A}\int_0^{\infty}\biggl(\frac{1-e^{-sp}}{sp}\biggr)^ne^{-k\sqrt{p}s}\frac{s^{n-1}}{(n-1)!}\,ds=\frac{e^Ap^{-n}}{(n-1)!}\int_0^{\infty} \frac{1}{s}(1-e^{-sp})^ne^{-k\sqrt{p}s}\,d s\\
&=\frac{e^Ap^{-n}}{(n-1)!}\int_0^{\infty} \frac{1}{\eta}(1-e^{-\eta})^n \exp\left({-\frac{k\eta}{\sqrt{p}}}\right)\,d \eta\\
& \leq \frac{e^Ap^{-n}}{(n-1)!}\int_0^{\infty} (1-e^{-\eta})^{n-1} \exp\left({-\frac{k\eta}{\sqrt{p}}}\right)\,d \eta \\
& = \frac{e^Ap^{-n}}{(n-1)!}\int_0^{1} (1-z)^{n-1} z^{\frac{k}{\sqrt{p}}-1}\,d z = \frac{e^Ap^{-n}}{(n-1)!} B(n, k/\sqrt{p}),
\end{align*} 
where 
$B(a_1, a_2) \equiv \int_0^{1} z^{a_1-1} (1-z)^{a_2-1}\,d z$ is the beta function, also known as the Euler integral of the first kind.
The beta function is closely connected to the gamma function
\begin{equation*}
    B(a_1, a_2)=\frac{\Gamma(a_1)\Gamma(a_2)}{\Gamma(a_1+a_2)}, \quad \text{with} \quad \Gamma(a) \equiv  \int_0^{\infty} z^{a-1} e^{-z}\,d z,
\end{equation*}
known as the Euler integral of the second kind, and the binomial coefficients
\begin{equation*}
\binom{n+k}{k} = \frac{(n+k)!}{n!k!}=\frac{n+k}{nk}\frac{\Gamma(n+k)}{\Gamma(n)\Gamma(k)} =\frac{n+k}{nk} \frac{1}{B(n,k)}.
\end{equation*}
By using these connections, we obtain the following bound:
\begin{align*}
\E[S^{\eps}(n,n+k)]
&=(n+k)_n\cdot P^{\eps}(n,n+k)\\
&\leq n! \cdot \frac{n+k}{nk} \frac{1}{B(n,k)} \cdot \frac{e^Ap^{-n}}{(n-1)!} B(n, k/\sqrt{p})\\
&= e^A \frac{(n+k)}{k} p^{-n}\frac{B(n, k/\sqrt{p})}{B(n, k)}\\
&\leq  e^A \frac{(n+k)}{k} p^{-n}=\exp(O(\log n))
\end{align*}
because $p=e^{-2\eps \lambda}\le 1$ and $\eps\lambda=O(n^{-1}\log n)$.
\end{proof}

\begin{lemma}\label{lem4} If $\eps \lambda =\omega(n^{-1}\log n)$, i.e. $\eps \lambda\gg n^{-1}\log n$, then $\E[S^{\eps}(n,n+k)]$ is super-polynomially large, uniformly for all $k\ge 1$. 
\end{lemma}

\begin{proof}
Since $\E[S^{\eps}(n,n+k)]$ is increasing in $\epsilon\lambda$, and $\epsilon\lambda=\omega(n^{-1}\log n)$, it is sufficient to focus only on $p=e^{-2\epsilon\lambda}=1-o(1)$ in our calculations below.

As before, we define the following variables:
\begin{equation*}
    s \equiv \sum_{i \in [n]}x_i,\quad s_j \equiv \sum_{i\neq j}x_i=s-x_j, \quad t \equiv \sum_{i \in [n]}x_i^2 , \quad \text{and} \quad t_j\equiv \sum_{i\neq j}x_i^2=t-x_j^2.
\end{equation*}

Restrict attention to the set $D$ of all non-negative $x=(x_1,x_2,\ldots,x_n)$ such that
\begin{align}
\label{set_D_imbalance}
\begin{split}
&s_1(n)  \leq s \leq s_2(n), \\
&s^{-1} x_i \leq \frac{2\log n}{n},\quad i \in [n],\\
&s^{-2} t \leq \frac{3}{n},
\end{split}
\end{align}
where $s_2(n)>s_1(n)$ are to be defined later. (As before, the two bottom inequalities are essentially dictated by the connection between the fractions $s^{-1}x_i$ and the lengths $L_i$ of random subintervals 
of $[0,1]$, as well as Lemma \ref{lem0}.) According to the first two conditions, we have
\begin{equation*}
    x_i \leq s\frac{2\log n}{n} \leq \frac{2s_2(n)\log n}{n} \equiv \sigma_n.
\end{equation*}
In what follows, we choose $s_2(n)$ such that $\sigma_n \to 0$ (to be verified later). Thus, for sufficiently large $n$, we have $x_i\leq\sigma_n<1$ and $D \subset [0,1]^n$. Consequently, for such large $n$,
\begin{multline*}
P^\epsilon(n,n+k)\geq \tilde{P}^\epsilon(n,n+k) \equiv  \int_{\bold x \in D}\!\left(\prod_{j \in [n]}\left(\int_0^1 \prod_{i \neq j} (1-px_iy_j)\,dy_j\right)\right)\\
\times\prod_{\ell\in [n]}(1-\sqrt{p}x_{\ell})^k\,d\bold x.
\end{multline*}
Since $1-\alpha=\exp(-\alpha-\alpha^2(1+O(\alpha))/2)$ as $\alpha \to 0$, we have: for each $j$,
\begin{align*}
    \prod_{i \in[n], i\neq j} \left(1-px_iy_j\right) 
    &= \exp\left(-py_js_j-p^2y^2_jt_j\frac{1+O\left(\sigma_n\right)}{2}\right)\\
    &\geq \exp\left(-py_js-\frac{2p^2y^2_jt}{3}\right).
\end{align*}
By proceeding as in the proof of Lemma~\ref{lem2}, and using $t/s^2\leq3/n$, we bound
\begin{equation*}
\prod_{j \in [n]} \left( \int_0^1 \prod_{i \neq j} (1 - p x_i y_j) dy_j \right) \\
\geq e^{-4} \left(\frac{1 - e^{-ps}}{ps} \right)^n. 
\end{equation*}

Furthermore,
\begin{align*}
    \prod_{\ell\in [n]}(1-\sqrt{p}x_{\ell})^k 
    &= \exp\left(-k\sqrt{p}s-kpt\frac{1+O\left(\sigma_n\right)}{2}\right)\\
    &\geq \exp\left(-k\sqrt{p}s-\frac{2kps^2}{3}\frac{t}{s^2}\right)\\
    &\geq \exp\left(-k\sqrt{p}s-\frac{2kps^2}{n}\right)
\end{align*}
Based on the previous calculations, we have
\begin{align*}
P^\epsilon(n,n+k)\geq \int_{\bold x \in D}\!e^{-4} \left(\frac{1 - e^{-ps}}{ps} \right)^n \exp\left(-k\sqrt{p}s-\frac{2kps^2}{n}\right)\,d\bold x.
\end{align*}
We then switch to new variables $(u,v_1,v_2,\ldots,v_{n-1})$ defined as follows:
\begin{equation*}
    u \equiv \sum_{i \in [n]} x_i = s \qquad \text{and} \qquad v_i \equiv
    s^{-1}x_i, \quad i \in [n-1].
\end{equation*}
We also define $v_n \equiv s^{-1}x_n=1-\sum_{i \in [n-1]}v_i$. The conditions \eqref{set_D_imbalance} then become
\begin{align}
\begin{split}
&s_1(n)  \leq u \leq s_2(n), \\
&v_i \leq \frac{2\log(n)}{n},\quad i \in [n],\\
&\sum_{i \in [n]}v_i^2\leq \frac{3}{n},
\end{split}
\end{align}
Importantly, for these new variables, the resulting region is the direct product of the range of $u$ and the range of $\bold v$. By using the key equality \eqref{A},  we get then
\begin{align*}
\tilde{P}^\epsilon(n)& \geq \frac{e^{-4}p^{-n}}{(n-1)!} \left( \int_{s_1(n)}^{s_2(n)} (1 - e^{-pu})^n \exp\left(-k\sqrt{p}u\right)\cdot \frac{\exp\left(-\frac{2kpu^2}{n}\right)}{u} du \right)\\
&\quad\times \Prob\left(L^+_n  \leq \frac{2\log(n)}{n} \quad \text{and} \quad U_n \leq \frac{3}{n}\right),
\end{align*}
where $L^+_n \equiv  \max_{i \in [n]} L_i$ and $U_n \equiv \sum_{i \in [n]}L_i^2$. By Lemma \ref{lem0},  the probability of the event on the right-hand side tends to one as $n \to \infty$. Thus, it remains to examine
\begin{equation}\label{int*} 
\begin{aligned}
\int^\star& \equiv \int_{s_1(n)}^{s_2(n)}  \frac{\exp\left(H(u)-\frac{2kpu^2}{n}\right)}{u} du,\\
 H(u)& \equiv n\log (1 - e^{-pu})-k\sqrt{p}u.
\end{aligned}
\end{equation}
for appropriately chosen $s_{1,2}(n)$. Note that
\begin{equation*}
H'(u)=\frac{np}{e^{pu}-1}-k\sqrt{p},\,\, H''(u)=-\frac{np^2e^{pu}}{(e^{pu}-1)^2}
=-\frac{np^2}{(e^{pu/2}-e^{-pu/2})^2}\ge-\frac{n}{u^2},
\end{equation*}
i.e., the function $H(u)$ attains its maximum at 
$s(n) \equiv \frac{\log \left(\frac{n+k/\sqrt{p}}{k/\sqrt{p}}\right)}{p}$.
We anticipate that the dominant contribution to the integral $\int^\star$ is obtained when we set  
\begin{equation*}
 s_{1,2}(n)=s(n)(1\mp\delta(n))
\end{equation*}
for some $\delta(n) \to 0$. And for $u\in [s_1(n),s_2(n)]$, we have $H''(u)=-\Theta(n/s^2(n))$. (Note that $\sigma_n=\frac{2s_2(n)\log n}{n}\to 0$ for $n \to \infty$, as required.) Further, 
\begin{equation*}
H(s(n))=J(k/\sqrt{p}), \quad J(z):=n\log\frac{n}{n+z}+z\log\frac{z}{n+z},
\end{equation*}
where $J'(z)=\log\tfrac{z}{n+z}$, $J''(z)=\tfrac{n}{z(n+z)}>0$. It follows that 
\begin{equation*}
H(s(n))\ge n\log\frac{n}{n+k}+k\log\frac{k}{n+k}-\tfrac{1-p}{1+\sqrt{p}}k\log\frac{n+k}{k}.
\end{equation*}
So, expanding $H(u)$ in powers of $z:=u-s(n)$, we see that the integral $\int^*$ (defined in \eqref{int*}) is at least of order
\[
\frac{\exp(H(s(n))-3.1kps^2(n)/n)}{s(n)}\int_{-\delta(n)s(n)}^{\delta(n)s(n)}\exp\bigl(-z^2\Theta(n/s^2(n)\bigr)\,dz.\\
\]
Here the integral equals
\[
\tfrac{s(n)}{\sqrt{n}}\int_{-\delta(n)\sqrt{n}}^{\delta(n)\sqrt{n}}e^{-\Theta(\eta^2)}\,d\eta=\Theta(\tfrac{s(n)}{\sqrt{n}}),
\]
provided that $\delta (n)\sqrt{n}\to\infty$, which we assume. Therefore the integral $\int^*$ is at least of order
\begin{multline*}
\frac{\exp(H(s(n))-3.1kps^2(n)/n)}{\Theta(\sqrt{n})}\\
=\frac{n^n k^k}{(n+k)^{n+k}\Theta(\sqrt{n})}\exp\Bigl(-\tfrac{1-p}{1+\sqrt{p}}k\log\frac{n+k}{k}-3.1kps^2(n)/n\Bigr).
\end{multline*}
So, $\tilde{P}^\epsilon(n)$ is at least of order 
\begin{equation*}
\frac{p^{-n}}{(n-1)!}\frac{n^n k^k}{(n+k)^{n+k}\sqrt{n}}\exp\Bigl(-\frac{1-p}{1+\sqrt{p}}k\log\frac{n+k}{k}-3.1kps^2(n)/n\Bigr).
\end{equation*}
Combining this estimate with 
\begin{equation*}
\E[S^\epsilon (n,n+k)] 
\geq (n+k)_n \cdot \tilde{P}^\epsilon(n),\quad (n+k)_n\ge c e^{-k}(n+k)^{n+k} e^{-n},
\end{equation*}
$c$ being an absolute constant, we obtain that $\E[S^\epsilon (n,n+k)]$ is at least of order
\begin{equation*} 
\exp\Bigl(n\log(1/p)-\tfrac{1-p}{1+\sqrt{p}}k\log\tfrac{n+k}{k}-3.1kps^2(n)/n\Bigr).
\end{equation*}
Here, by the definition of $s(n)$,
\begin{equation*}
kps^2(n)/n=\frac{1}{\sqrt{p}}\cdot\tfrac{k}{n\sqrt{p}}\log^2\bigl(\frac{1+k/n\sqrt{p}}{k/n\sqrt{p}}\bigr)=O(1/\sqrt{p}),
\end{equation*}
since $\sup_{z>0} z\log^2\bigl(\frac{1+z}{z}\bigr)<\infty$.  In addition, $\frac{k}{n}\log\tfrac{n+k}{k}\le \frac{k}{n}\cdot \tfrac{n}{k}=1$, and $\log(1/p)>1-p$. Therefore $\E[S^\epsilon (n,n+k)]$ is at least of order
\begin{multline*}
\frac{1}{\sqrt{n}}\exp\Bigl(n\bigl(\log(1/p)-\frac{1-p}{1+\sqrt{p}}\bigr)+O(1/\sqrt{p})\Bigr)\\
\ge\frac{1}{\sqrt{n}} \exp\bigl(n(1-p)\frac{\sqrt{p}}{1+\sqrt{p}}+O(1/\sqrt{p})\bigr),
\end{multline*}
which is super-polynomially large for $p=e^{-2\eps\lambda}$, $\eps\lambda\to 0$, $\eps\lambda=\omega \tfrac{\log n}{n}$, $\omega=\omega(n)\to\infty$. This completes the proof of Lemma \ref{lem4}.
\end{proof}
Hence, the proof of Theorem \ref{thm1} for $k\ge 1$ is complete as well.

\bigskip\noindent
{\bf Acknowledgment.\/} We thank Federico Echenique, Haluk Ergin, Michael Ostrovsky, Leeat Yariv for helpful comments.

\end{document}